\documentclass{amsart}
\usepackage{graphicx} 

\usepackage{fullpage}
\usepackage{enumerate}
\usepackage{enumitem} 
\usepackage{float}

\usepackage[bookmarks=true, colorlinks=true, linkcolor=blue!50!black,
citecolor=orange!50!black, urlcolor=orange!50!black, pdfencoding=unicode]{hyperref}
\usepackage{color}

\usepackage{changepage}   
\newtheorem*{theorem*}{Theorem}

\newtheorem{maintheorem}{Theorem}[section]

\newtheorem{theorem}{Theorem}[section]
\newtheorem{lemma}[theorem]{Lemma}
\newtheorem{proposition}[theorem]{Proposition}

\theoremstyle{definition}
\newtheorem{definition}[theorem]{Definition}
\newtheorem{example}[theorem]{Example}

\newtheorem*{Acknowledgements}{Acknowledgements}

\newtheorem{remark}[theorem]{Remark}

\usepackage{tikz}									
\usetikzlibrary{matrix}
\usetikzlibrary{patterns}
\usetikzlibrary{matrix}
\usetikzlibrary{positioning}
\usetikzlibrary{cd}

\def\semicolon{;}
\def\applytolist#1{
    \expandafter\def\csname multi#1\endcsname##1{
        \def\multiack{##1}\ifx\multiack\semicolon
            \def\next{\relax}
        \else
            \csname #1\endcsname{##1}
            \def\next{\csname multi#1\endcsname}
        \fi
        \next}
    \csname multi#1\endcsname}

\usepackage{amssymb}
\def\calc#1{\expandafter\def\csname b#1\endcsname{{\mathbb #1}}}
\applytolist{calc}QWERTYUIOPLKJHGFDSAZXCVBNMqwertyuiopasdfghjklzxcvbnm;

\def\calc#1{\expandafter\def\csname bf#1\endcsname{{\mathbf #1}}}
\applytolist{calc}QWERTYUIOPLKJHGFDSAZXCVBNMqwertyuiopasfghjklzxcvbnm;

\def\calc#1{\expandafter\def\csname c#1\endcsname{{\mathcal #1}}}
\applytolist{calc}QWERTYUIOPLKJHGFDSAZXCVBNM;
\def\calc#1{\expandafter\def\csname cal#1\endcsname{{\mathcal #1}}}
\applytolist{calc}QWERTYUIOPLKJHGFDSAZXCVBNM;

\def\calc#1{\expandafter\def\csname s#1\endcsname{{\mathscr #1}}}
\applytolist{calc}QWERTYUIOPLKJHGFDSAZXCVBNM;

\def\calc#1{\expandafter\def\csname f#1\endcsname{{\mathfrak #1}}}
\applytolist{calc}QWERTYUIOPLKJHGFDSAZXCVBNMqwertyuopasdfghjklzxcvbnm;

\def\calc#1{\expandafter\def\csname tb#1\endcsname{{\text{\textbf{#1}}}}}
\applytolist{calc}QWERTYUIOPLKJHGFDSAZXCVBNMqwertyuiopasdfghjklzxcvbnm;

\def\calc#1{\expandafter\def\csname u#1\endcsname{{{\textrm{#1}}}}}
\applytolist{calc}QWERTYUIOPLKJHGFDSAZXCVBNMqwertyuiopasdfghjklzxcvbnm;

\def\calc#1{\expandafter\def\csname cl#1\endcsname{{{\overline{#1}}}}}
\applytolist{calc}QWERTYUIOPLKJHGFDSAZXCVBNMqwertyuiopasdfghjklzxcvbnm;

	\usepackage{mathrsfs}

\newcommand{\an}{\textrm{an}}

\newcommand{\R}{\mathbb{R}}

\newcommand{\trop}{\textrm{trop}}
\newcommand{\Trop}{\textrm{Trop}}
\newcommand{\SL}{\textrm{SL}}
\newcommand{\GL}{\textrm{GL}}

\newcommand{\spec}{\textrm{Spec}}

\newcommand{\val}{\textrm{val}}
\newcommand{\Hom}{\textrm{Hom}}

\newcommand{\Sigmabar}{\overline{\Sigma}}

\DeclareMathOperator{\codim}{codim}
\DeclareMathOperator{\Star}{Star}

\title{Spherical tropical curves are balanced} 

\date{}

\author{Desmond Coles}
\address{The University of Texas at Austin,
Department of Mathematics,
Austin, TX 78712,
USA
}
\email{dcoles@utexas.edu}

\begin{document}

\begin{abstract}
One of the characterizing features of tropicalizations of curves in an algebraic torus is that they are {balanced}. Tevelev and Vogiannou introduced a spherical tropicalization map for spherical homogeneous spaces $G/H$, where $G$ is a reductive group. This map generalizes the tropicalization map for algebraic tori. We prove a balancing condition for spherical tropicalizations of curves in $G/H$ that generalizes the balancing condition for tropicalizations of curves contained in an algebraic torus. We give examples and describe the relationship with Andreas Gross's balancing condition for tropicalizations of subvarieties of toroidal embeddings.
\end{abstract}

\maketitle

\tableofcontents

\section{Introduction}

Let $k$ be an algebraically closed and trivially valued field of characteristic 0, let $G$ be a (connected) reductive group over $k$ such as an algebraic torus $T$, $\GL_n$, or $\SL_n$. Let $G/H$ be a spherical homogeneous space of dimension $d$. Recall that a variety equipped with a $G$ action is \textit{spherical} if there is a Borel subgroup $B$ in $G$ with an open orbit. Toric varieties, (partial) flag varieties (in particular Grassmannians), the basic affine space $G/U$, and symmetric varieties (including reductive groups themselves and their wonderful compactifications) are all spherical varieties. Many facts about toric varieties generalize to spherical varieties, in particular there is a combinatorial classification of spherical varieties, often called the `Luna-Vust theory', that generalizes the correspondence between toric varieties and fans \cite{LunaVust,Knop}.
As in \cite{TVSpherTrop}, \cite{NashGlobal}, and \cite{ColesAnalytic} there is a tropicalization map for spherical varieties given as follows. Let $G/H$ be a spherical homogeneous space. There is a finite dimensional vector space $N_{\R}(G/H)$ associated to $G/H$. The set of $G(k)$-invariant valuations on $k(G/H)$, which we denote $\calV$, is identified with a finitely generated convex cone in $N_{\R}(G/H)$. 

Let $K$ be an algebraically closed non-trivially valued extension of $k$ and let $\gamma \in G/H(K)$. 
Then define $\trop_G\colon G/H(K)\rightarrow \cV$ by  $\gamma \mapsto \textrm{val}_{\gamma} $ where:
\[
 \textrm{val}_{\gamma}(f)= \textrm{val}_K(g\cdot f(\gamma)) \textrm{ for $g\in U_f$}
\]
here $f$ is an element of $k(G/H)$ and $U_f$ is a Zariski dense open in $G/H(k)$ such that the right hand side is constant for $g\in U_f$. When $G=T$ is an algebraic torus, and $H$ is the trivial subgroup, this is the Kajiwara-Payne tropicalization map as in \cite{PayneLimit}. Several foundational results from the toric case have been extended to the spherical case: Nash showed that the spherical tropicalization map extends from homogeneous spaces to arbitrary spherical varieties in the same way as the tropicalization map for tori extends to toric varieties  \cite{NashGlobal}, Tevelev and Vogiannou showed that spherical tropical compactifications exist  \cite{TVSpherTrop}, Kaveh and Manon introduced a Gröbner theoretic approach to spherical tropicalization as well as a notion of amoeba \cite{KMGrob}, and in \cite{ColesAnalytic} it is shown that the tropicalization map extends to a continuous map from the Berkovich analytification $\trop_G\colon G/H^{\an}\rightarrow \cV$, and that this map admits a section. Readers may also want to see \cite{KMSurvey}, \cite{NashPreprint}. \cite{VogThesis}, and \cite{ColesUlirsch} for more background on spherical tropicalization.
A fundamental result of tropical geometry is that when $Z$ is an irreducible $n$-dimensional subvariety of a toric variety, the tropicalization of $Z$ has the structure of a purely $n$-dimensional \textit{balanced} weighted fan which is connected in codimension 1 \cite{TropElimTheory, MaclaganSturmfels,GublerGuide}.
Spherical tropicalizations of subvarieties do not in general display these qualities (see Example \ref{example_GLNGLNGLN}), though it is known that the spherical tropicalization of a subvariety is the support of a rational fan in $N_{\R}(G/H)$ \cite[Theorem 1]{TVSpherTrop}.
In this article we will describe the balancing condition for spherical tropicalizations of curves defined over a trivially valued field.

Let $\Trop_T(C)$ be the tropicalization of a curve $C$ in an algebraic torus $T$ of dimension $d$. Then $\Trop_T(C)$ is the support of a 1-dimensional polyhedral complex in $\bR^d$, in particular $\Trop_T(C)$ has a unique fan structure. For each 1-dimensional cone $\sigma$ in this fan let $v_{\sigma}$ be the primitive integer vector that spans $\sigma$. A primitive integer vector is one with integer coordinates that do not have a common divisor. Tropicalization attaches a weight $m_{\sigma}$ to each cone $\sigma$ (as described in the following paragraph), and then we have:
\[
\sum\limits_{\sigma}m_{\sigma}v_{\sigma}=0.
\]
This property of tropical curves is a crucial part in using tropical geometry to solve enumerative problems in algebraic geometry, see for example Mikhalkin's correspondence theorem \cite{Mikhalkin1}, or the applications to log geometry in \cite{DhruvStable}.

The underlying machinery of balancing for spherical tropical curves is the same as in the case of toric varieties. For $C$ in an algebraic torus $T$ the weights of $\Trop_T(C)$ are given by intersections of a tropical compactification of $C$ with the boundary of the associated toric variety \cite{TropElimTheory,GrossIntersection}. Then it can be shown that the tropicalization in fact contains the same data as the Chow cohomology class determined by $C$ in the toric variety and it follows from the description of Chow cohomology of toric varieties in \cite{FSToric} that the curve must be balanced. The result in \cite{FSToric} relies upon the earlier work in \cite{FMSS} that gives a description of the Chow cohomology for spherical varieties. We apply the results in \cite{FMSS} in a manner analogous to those in \cite{TropElimTheory} to arrive at a balancing condition in the spherical case. 

Let $C$ be a curve over $k$ in $G/H$. The set $\Trop_G(C)$ has the structure of a fan in $\cV$ and by \cite{TVSpherTrop} this fan determines a tropical compactification of $C$. Let $X$ be a smooth $G$-equivariant toroidal embedding of $G/H$ such that the fan associated to $X$ contains $\Trop_G(C)$ as a subfan. 
Let $\overline{C}$ be the closure of $C$ in $X$.
Given a cone $\sigma$ in $\Trop_G(C)$ we can define the \textit{weight} associated to $\sigma$, denoted $m_{\sigma}$, to be
\[
m_{\sigma}:=\deg\left([\overline{C}]\cdot [D_{\sigma}]\right)
\]
where $D_{\sigma}$ is the boundary divisor in $X$ associated to $\sigma$. For the balancing condition we need correction factors given by $B$-invariant prime divisors in $G/H$, let $E_j$ be such a divisor (these divisors give rise to the \textit{colors} in the colored fans). Then we define the \textit{colored weight}, denoted $m_j$, to be
\[
m_j:=\deg\left([\overline{C}]\cdot [\overline{E_{j}}]\right)
\]
where $\overline{E_j}$ is the closure of $E_j$ in $X$. Each such divisor $E_j$ determines a vector $v_j$ in $N_{\R}(G/H)$ (See Section \ref{Section_TheLunaVustTheory}). Let $v_{\sigma}$ be the primitive integer vector that generates the cone $\sigma$.
In Section \ref{Section_balanacingTWO} we show that $X$ exists and that the numbers $m_{\sigma}$ and $m_j$ are independent of the choice of $X$. Then we have the following balancing condition.
\begin{maintheorem}\label{mainthm_balancingforcurves}
    Let $G/H$ be a spherical homogeneous space and let $C$ be a curve in $G/H$ (defined over $k$). Then $\Trop_G(C)$ satisfies the following balancing condition:
    \[
    \sum\limits_{\sigma}m_{\sigma}v_{\sigma}=-\sum\limits_j m_j v_j.
    \]
\end{maintheorem}
\noindent
When $G$ is a torus and $H$ is trivial there are no $B$-stable divisors $E_j$ in $G/H$, as $B=G=T$, so this specializes to the balancing condition for curves in tori. See Section \ref{Section_Examples} for examples of the balancing condition.

A key idea in this proof is that we can work with toroidal spherical varieties, i.e. spherical varieties $X$ such that the embedding of the open $G$-orbit, $G/H$, into $X$ is toroidal. This is equivalent to the colored fan associated to $X$ having no colors and as we discuss in Section \ref{Section_toroidalembeddings} spherical tropicalization agrees with the tropicalization map for toroidal embeddings.
Furthermore, in \cite{GrossIntersection} there is a balancing condition for tropicalizations of subvarieties of toroidal embeddings, and in Section \ref{section_closingremarks} we discuss how our balancing condition relates to the one presented in \cite{GrossIntersection}.

The paper is organized as follows. First we present the Luna-Vust theory of spherical varieties which describes spherical varieties in a combinatorial manner that generalizes the correspondence between toric varieties and fans. This section contains important notation. We then review spherical tropicalization.
In the following section we review the results on tropical compactifications in \cite{TVSpherTrop}. Finally, we prove the balancing condition. In the last section we include some discussion of how the ideas in this article may relate to a general structure theory for spherical tropical varieties. We assume the reader has some familiarity with tropical geometry and Berkovich geometry throughout.

\begin{Acknowledgements}
I want to thank María Angélica Cueto, Alex Fink, Andreas Gross, Eric Katz, Gary Kennedy, Amy Li, Isaac Martin, Evan Nash, Sam Payne, Dhruv Ranganathan, Kris Shaw, Bernd Siebert, Frank Sottile, and Martin Ulirsch for helpful conversations. In particular I would like to thank my advisor Sam Payne for his support and encouragement while writing this article, as well as the math department at The University of Texas at Austin for providing a welcoming and supportive environment to do mathematics. I was supported by NSF DMS–2302475 and NSF DMS–2053261 while completing this work.
\end{Acknowledgements}

\section{The Luna-Vust theory of spherical varieties}\label{Section_TheLunaVustTheory}
In this section we review the Luna-Vust classification of spherical varieties. This is a generalization of the correspondence between toric varieties and fans in \cite{FultonToricBook}. Recall that the correspondence between toric varieties and fans can be phrased as follows: if $T$ is an algebraic torus then $T$-equivariant partial compactifications of $T$ correspond to rational fans in $N_{\bR}=N\otimes_{\bZ}\bR$ (where $N$ is the cocharacter lattice of $T$). For a spherical homogeneous space $G/H$ the Luna-Vust theory gives a correspondence between `colored fans' and $G$-equivariant partial compactifications of $G/H$. By a partial compactification we mean a $G$-variety $X$ with a $G$-equivariant open embedding of $G/H$ into $X$, we call such a variety $X$ a `$G/H$-embedding' (note that $X$ is automatically spherical and any spherical variety has an open $G$-orbit that is $G$-equivariantly isomorphic to $G/H$ for some $H$).
We will follow the exposition of \cite{Knop}. First, we present some data associated to $G/H$ which allows us to state the correspondence precisely.

Let $k(G/H)$ be the function field of $G/H$. Define the group of $B$\textbf{-semi-invariant} rational functions, denoted $k(G/H)^{(B)}$, to be the set
\[
\{f\in k(G/H) \mid \textrm{there is a character of $B$, denoted $\chi_f$, such that for $b\in B(k)$ we have $b\cdot f =\chi_f(b)f$}\}.
\]
There is a homomorphism of groups from $k(G/H)^{(B)}$ to the character lattice of $B$ given by $f\mapsto \chi_f$, denote the image by $M(G/H)$. We call $M(G/H)$ the \textbf{weight lattice} of $G/H$. Let $N(G/H)$ be the dual lattice to $M(G/H)$, and set $N_{\bR}(G/H)=N(G/H)\otimes_{\bZ} \bR$. Let $\calV$ be the collection of $G(k)$-invariant valuations on $k(G/H)$. Then there is a map $\varrho\colon \calV \rightarrow N_{\bR}(G/H)$ given by $\varrho(\val)(\chi_f)=\val(f)$. Not only is this map well-defined but it is injective and the image is a finitely generated rational convex cone in $N_{\bR}(G/H)$ \cite[Corollary 1.8, Corollary 5.3]{Knop}. We identify $\calV$ with its image in $N_{\bR}(G/H)$. We call $\cV$ the \textbf{valuation cone} of $G/H$. There is a finite set of codimension 1 $B$-stable subvarieties $\calD=\{E_1,\ldots E_r\}$, and for each element we can also define a rational element of $N_{\R}(G/H)$ given by $\varrho(\val_{E_i})$ where $\varrho$ is defined the same as above. We say that $\calD$ is the \textbf{palette} of $G/H$. More generally if $X$ is a $G/H$-embedding we write $\calD(X)$ for the codimension 1 $B$-stable subvarieties. If $Y$ is a $G$-orbit in $X$ then we write $\calD_Y(X)$ for the elements of $\calD(X)$ such that $Y\subseteq D$. Because $X$ is spherical $ \calD(X)$ is finite \cite[Theorem 1]{Vinburg}. We can now define colored cones and colored fans.
\begin{definition}
    A \textbf{colored cone} is a pair $(\sigma, \calF)$, where $\sigma\subseteq N_{\bR}(G/H)$ is a strictly convex cone and $\calF\subseteq \calD$, and they satisfy the following:
    \begin{enumerate}[label=\textbf{CC\arabic*}, leftmargin=2cm, topsep=0cm, itemsep=0cm]
    \item $\sigma$ is generated by $\varrho(\cF)$ and finitely many rational elements of $\cV$
    \item The relative interior of $\sigma$ intersects $\cV$
    \item $0\notin \varrho(\cF)$.
\end{enumerate}
\end{definition}
\noindent
The elements of $\calF$ are referred to as the \textbf{colors} of the colored cone. A \textbf{colored face} of a colored cone $(\sigma, \calF)$ is a colored cone $(\sigma',\calF')$ such that $\sigma'$ is a face of $\sigma$ and $\calF'=\calF\cap \varrho^{-1}(\sigma')$. 
\begin{definition}
    A \textbf{colored fan} $\fF$ is a finite collection of colored cones such that:
\begin{enumerate}[label=\textbf{CF\arabic*}, leftmargin=2cm, topsep=0cm, itemsep=0cm]
    \item if $(\sigma,\calF)\in\fF$ then every colored face of $(\sigma,\calF)$ is contained in $\fF$.
     \item For each point of $v\in \calV$ there is at most one colored cone $(\sigma,\calF)$ such that $v$ is contained in the relative interior of $\calV$.
\end{enumerate}
\end{definition}

\begin{remark}
    There are some subtleties about colored fans to address. First, \cite{Knop} considers colored cones $(\sigma,\calF)$ where there may be $E\in \calF$ such that $\varrho(E)=0$. When $0\notin \calF$ and $\sigma $ is strictly convex, Knop says the colored cone $(\sigma, \calF)$ is `strictly convex'. We do not need to consider cones where $0\in \calF$, so we do not. That is, all of our colored cones are assumed to be `strictly convex' in Knop's sense. Secondly, the underlying collection of cones need not form a fan. It is true however that if the colored fan contains no colors (the fan is `toroidal' in other words) then the underlying set of cones does form a fan. The failure to form a fan is a result of the fact that colored cones need only to intersect along faces in the relative interior of $\calV$. The failure of the colored fan to be an actual fan has a geometric meaning \cite{Gagliardi}.
\end{remark}

Let $X$ be a $G/H$-embedding and let $Y\subseteq X$ be a $G$-orbit. Then we can construct a colored cone as follows. Define
\begin{align*}
    \calB_Y(X)&:= \{D\in \calD_Y(X) \mid \textrm{such that $D$ is $G$-stable}\} \\
    \calF_Y(X)&:= \{D\in \calD_Y(X) \mid \textrm{such that $D$ is \emph{not} $G$-stable}\}.
\end{align*}
Then let $\sigma_Y(X)$ be the cone generated by $\varrho(\calB_Y(X))$ and $\varrho(\calF_Y(X))$. Note that $\calF_Y(X)$ is naturally identified with a subset of $\calD$ because elements of $\calD(X)$ which are not $G$-stable must be the closure of some element of $\calD$. We say that $G/H$-embedding is \textbf{simple} if there is a unique closed $G$-orbit.
\begin{theorem}\cite[Theorem 3.1]{Knop}
    Let $X$ be a $G/H$-embedding with $G$-orbit $Y$. Then we have that the pair $(\sigma_Y(X),\calF_Y(X))$ is a colored cone. For $X$ a simple embedding the map $X\mapsto (\sigma_Y(X),\calF_Y(X))$ (where $Y$ is the closed $G$-orbit), gives a bijection between simple $G/H$-embeddings and colored cones in $N_{\bR}(G/H)$.
\end{theorem}
\noindent
Note that if $G=G/H=T$ is an algebraic torus then simple $T$-embeddings are exactly affine toric varieties and this theorem specializes to the correspondence between affine toric varieties and cones.
\begin{theorem}\cite[Theorem 3.3]{Knop}\label{thm_coloredfans}
    Define $\fF(X)=\{(\sigma_Y(X),\calF_Y(X)) \mid Y \textrm{ is a $G$-orbit of $X$})\}$. Then $\fF(X)$ is a colored fan and the association $X\mapsto \fF(X)$ is a bijection between isomorphism classes of $G/H$-embeddings and colored fans in $N_{\bR}(G/H)$.
\end{theorem}

\begin{remark}
    To be precise, the association of a spherical variety to its fan is functorial and defines an equivalence of the two categories. In this article we will not make extensive use of this though so we will not discuss this in detail. Though when the fan $\fF(X)$ has no colors a morphism of colored fans is just a morphism of fans and such a morphism will induce a $G$-equivariant morphism of the corresponding emebddings. 
\end{remark}

We also have an orbit-cone correspondence for $G/H$-embeddings and we can describe the colored fan of an orbit closure using the colored fan $\fF(X)$. Let $Y$ be a $G$-orbit of $X$ and let $\overline{Y}$ be the closure in $X$.
\begin{proposition}\cite[Lemma 3.2]{Knop}
    The map $Z\mapsto (\sigma_Z(X), \calF_Z(X))$ defines a bijection $G$-orbits $Z$ such that $\overline{Z}\supseteq Y$ and colored faces of $(\sigma_Y(X),\calF_Y(X))$. In particular there is a bijection between $G$-orbits and colored cones of the colored fan of $X$.
\end{proposition}
\noindent
Furthermore we have that if $Y$ is a spherical homogeneous space, then $\overline{Y}$ is a $Y$-embedding, and  we can describe the colored fan $\fF(Y)$. Let $Y$ correspond to the colored cone $(\sigma,\calF)\in\fF(X)$, i.e. $(\sigma_Y(X),\calF_Y(X))=(\sigma,\calF)$. Then define $M(\sigma):=\sigma^{\perp}\cap M(G/H)\subseteq M_{\R}(G/H)$ is a lattice of characters and it consists of characters corresponding to rational functions which do not have poles or zeros on $Y$. So restriction from $X$ to $Y$ gives a map $M(\sigma)\rightarrow M(Y)$ and this map is an isomorphism by Theorem 6.3 of \cite{Knop}. So $M(\sigma)=M(Y)$. Notice that we have an inclusion $\iota \colon M(Y)\hookrightarrow M(G/H)$. Furthermore, if $N(\sigma)=\Hom(M(\sigma),\bZ)$ then it follows that $N(\sigma)=N(Y)$. The map $\iota$ induces a surjection $\iota^*\colon N_{\R}(G/H)\rightarrow N_{\R}(Y)$. Set
\[
\calD_{\iota}:=\{D\in \calD(G/H) \mid \textrm{$D\cap Y$ is a divisor of $Y$}\}
\]
\[
\textrm{Star}(\sigma,\calF):=\{(\iota^*(\tau),\calE\cap \calD_{\iota}) \mid (\tau,\calE)\in \mathfrak{F}(X) \textrm{ and $(\tau,\calE)$ is a colored face of $(\sigma,\calF)$}\}.
\]
If the set $\calF$ is empty we may write $\textrm{Star}(\sigma)$ for convenience.
Then we have the following.
\begin{theorem}\cite[Theorem 4.5]{Knop}\label{thm_FanOfOrbit}
     We have that $M(Y)=M(\sigma)$, $N(Y)=N(\sigma)$, and $\calD(Y)=\calD_{\iota}$. Furthermore $\mathfrak{F}(\overline{Y})$ is $\Star(\sigma,\calF)$.
\end{theorem}

We say that a $X$ is a \textbf{toroidal spherical embedding} if $\mathfrak{F}(X)$ contains no colors. Recall that an open embedding $U\hookrightarrow X$ is a `toroidal embedding' if it is étale locally given by the embedding of a torus into a toric variety. 
\begin{proposition}\cite[Proposition 7.13]{ColesUlirsch}\label{thm_toristor}
    The $G/H$-embedding $X$ is a spherical toroidal embedding if and only if the inclusion $G/H\hookrightarrow X$ is a toroidal embedding.
\end{proposition}
\noindent
For the remainder of this paper we will almost entirely work with toroidal $G/H$-embeddings. Note that a colored fan with no colors is just a fan supported in $\calV$, and thus by Proposition \ref{thm_toristor} and Theorem \ref{thm_coloredfans} we have that toroidal spherical embeddings correspond to fans that are supported in $\calV$. Furthermore by \cite[Proposition 3.3.2]{Perrin} the $G$-orbits of a toroidal embedding are exactly the toroidal strata of the toroidal embedding. In the following section we will see that the spherical tropicalization map agrees with the tropicalization map for toroidal embeddings and so the fan in $\calV$ is exactly the cone complex of the toroidal embedding.

\begin{remark}
    Given two toroidal $G/H$-embeddings $X$ and $X'$ with associated fans $\Delta$ and $\Delta'$ if for any cone $\sigma\in \Delta$ the identity map on $N_{\R}(G/H)$ maps $\sigma$ into some other cone $\sigma'\in \Delta'$ we have that this map defines a $G$-equivariant morphism $X\rightarrow X'$ such that the following diagram commutes:
\[
\begin{tikzcd}
G/H \arrow[d, hookrightarrow] \arrow[dr, hookrightarrow]  & \\
X \arrow[r] &  X'.
 \end{tikzcd}
\]
Conversely, any such morphism of $X\rightarrow X'$ induces the identity on $N_{\R}(G/H)$ and will map cones of $\Delta$ to cones of $\Delta'$.
\end{remark}

\section{Spherical tropicalization}\label{Section_SphericalTropicalization}

Tevlev and Vogiannou introduced a tropicalization map for spherical varieties \cite{TVSpherTrop}. The map is given as follows. Let $\gamma \in G/H(K)$ where $K$ is an algebraically closed and non-trivially valued field extension of $k$. Then define $\trop_G\colon G/H(K) \rightarrow \cV$ by $\gamma \mapsto \val_{\gamma}$ where
\[
\val_{\gamma}(f) = \val_K(g\cdot f(\gamma) ) \textrm{ for $g\in U_f\subseteq G(k)$}.
\]
Here $U_f$ is a Zariski dense open of $G(k)$ where the right hand side is constant for $g\in U_f$. The image of this map is dense in $\cV$ because $K$ is algebraically closed. This map agrees with the Kajiwara-Payne tropicalization map when $G$ is a torus. Nash extends this map to any $G/H$-embedding, in the same manner that the tropicalization map is extended from tori to toric varieties \cite{NashGlobal}. If $X$ is a $G/H$-embedding then $X$ is the disjoint union of some finite collection of $G$-orbits $Y_1,\ldots Y_m$, each of which is spherical \cite[Corollary 2.2]{Knop}, and for each there is an associated valuation cone $\calV(Y_i)$ and tropicalization map $\trop_G\colon Y_i(K)\rightarrow \calV(Y_i)$. We will now describe how to glue these maps together. We start with the case when $X$ is simple; let $(\sigma,\calF)$ be the corresponding colored cone. 

Consider the space $\Hom(\sigma^\vee,(-\infty,\infty])$,
here $\sigma^{\vee} : =\{m\in M_{\R} \mid u(m)\geq 0 \textrm{ for any $u\in \sigma^{\vee}$}\}$ and $\Hom$ denotes monoid homomorphisms that preserve scaling by elements of $\R$. 
 Let $u\in \Hom(\sigma^{\vee},(-\infty,\infty])$, the set where $u$ is not equal to $\infty$ is equal to $\tau^{\perp}\cap \sigma^{\vee}$ for some $\tau$ a face of $\sigma$. So $u$ defines an element of $\Hom(M_{\R}(\tau)\cap \sigma^{\vee},(-\infty, \infty))\cong N_{\R}(\tau)$. This gives a bijection:
\[
\Hom(\sigma^{\vee},(-\infty,\infty])\rightarrow \bigsqcup_{\textrm{$\tau$ a face of $\sigma$}}N_{\R}(\tau)
\]
This defines a partial compactification of $N_{\R}(G/H)$ as $N_{\R}(G/H)=N_{\R}(0)$.
By Theorem \ref{thm_FanOfOrbit} we have that $N_{\R}(\tau)$ is equal to $N_{\R}(Y_i)$ where $Y_i$ is the $G$-orbit corresponding to $\tau$, and so $\calV(Y_i)\subseteq N_{\R}(\tau)$. So then $\Trop_G(X)$ is defined to to be the subspace of $\Hom(\sigma^{\vee},(-\infty,\infty])$ defined by the disjoint union of the valuation cones $\calV(Y_1),\ldots \calV(Y_m)$. When $X$ is not necessarily simple we observe that we can cover $X$ by simple varieties, and because the interseciton of two simple spherical subvarieties will be simple, we glue the tropicalizations along the tropicalization of the intersection. The tropicalization maps will glue to form a continuous map:
\[
\trop_G\colon X(K) \rightarrow  \Trop_G(X)
\]
which we call the \textbf{spherical tropicalization map}.

This tropicalization map extends to a map from the Berkovich analytification of $X$, which we denote by $X^{\an}$, \cite{ColesAnalytic}. The spherical tropicalization map admits a section $J_G\colon \Trop_G(X)\hookrightarrow X^{\an}$ and the composition $J_G\circ \trop_G$ defines a retraction map $\bfp_G\colon X^{\an}\rightarrow X^{\an}$, in fact there is a strong deformation retraction from $X^{\an}$ onto $J_G(\trop_G(X))$ \cite[Theorem A]{ColesAnalytic}. There is also a Gröbner theory for spherical varieties as introduced by Kaveh and Manon \cite{KMGrob}, this theory agrees with the above approach to tropicalization, i.e. there is a `fundamental theorem of spherical tropical geometry' \cite[Theorem 4]{KMSurvey}. We will not need Gröbner theoretic tools so we will not introduce this here.

 Tevelev introduces `tropical compactifications' of subvariety of tori in \cite{TevelevCompactifications}. These compactifications of subvarieties arise from toric varieties containing $T$ and have certain desirable properties for intersection theoretic computations. Tropical compactifications are an important ingredient in balancing for tropicalizations in the toric case \cite{TropElimTheory}. In \cite{TVSpherTrop} Tevelev and Vogiannou introduce tropical compactifications for subvarieties of a spherical homogeneous space and prove that such compactifications exist.
Let $Z\subseteq G/H$ be a subvariety, and let $X$ be a $G/H$-embedding. Let $\overline{Z}$ be the closure of $Z$ in $X$. Then we say that $\overline{Z}$ is a \textbf{tropical compactification} of $Z$ if $\overline{Z}$ is proper and the mutliplication map $m\colon G \times \overline{Z}\rightarrow X$ is faithfully flat.
\begin{theorem}\cite[Theorem 1]{TVSpherTrop}\label{thm_tropcompact}
    Let $Z\subseteq G/H$ be a subvariety. Then a tropical compactification of $Z$ in a spherical $G/H$-embedding exists. If $X$ is the associated spherical variety of a tropical compactification, and $\Delta$ is the fan associated to $X$, then $\Trop_G(Z)$ is equal to the support of $\fF(X)$.
\end{theorem}

\noindent
See Section \ref{Section_Examples} for examples of spherical tropicalization.

\section{Tropicalization of toroidal spherical embeddings}\label{Section_toroidalembeddings}

Let $X$ be a proper toroidal $G/H$-embedding. Then there are a priori, two tropicalization maps from $X$: there is the spherical tropicalization map $\trop_G\colon X^{\an}\rightarrow \Trop_G(X)$ and the tropicalization map induced by the toroidal structure $\trop_{\log}\colon X^{\an}\rightarrow \overline{\Sigma}_X$ \cite{UlirschLoTrop}. The toroidal tropicalization map admits a section $J_{\log}\colon \Sigmabar_{X}\rightarrow X^{\an} $ and there is a retraction $\bfp_X\colon X^{\an}\rightarrow X^{\an}$ given by $J_{\log}\circ \trop_{\log}$. These two tropicalization maps agree.
\begin{theorem}\cite[Theorem C]{ColesUlirsch}\label{thm_sphertroplogtrop}
    The image of the sections $J_{G}$ and $J_{\log}$ are equal and the retraction maps $\bfp_G$ and $\bfp_X$ are equal.
\end{theorem}
In particular we have a canonical identification of the tropicalizations $J_G^{-1}\circ J_{\log}\colon \overline{\Sigma}_X\rightarrow \Trop_G(X)$ such that the following diagram commutes:
\[
\begin{tikzcd}
X^{\an} \arrow[d, "\trop_{\log}"] \arrow[drr, "\trop_G"]  & \\
\overline{\Sigma}_X \arrow[rr, "\sim"] \arrow[u, bend left, "J_{\log}"] & &  \Trop_G(X) \arrow[ull, swap,  bend right, "J_G"]. \\
 \end{tikzcd}
\]
\begin{remark}
    The assumption that $X$ is proper is not especially serious for our purposes: given a toroidal compactification $X$ of $G/H$ there is some $G$-equivariant compactification of $X$ by \cite[Theorem C]{Sumihiro} and this compactification can be modified to be toroidal by taking each colored cone $(\sigma,\calF)$ in the associated colored fan, intersecting $\sigma$ with $\calV$, and removing $\calF$. The resulting fan will define a proper toroidal $G/H$-embedding which contains $X$ as an open subspace. Then we can restrict the tropicalization map to a subspace of this compactification to apply Theorem \ref{thm_sphertroplogtrop} to $X$.
    We introduced the condition that $X$ is proper to avoid having to introduce some definitions from Berkovich geometry that we would not otherwise need in this paper. Readers familiar with Thuillier's $\beth$-functor should note that $\bfp_X$ is only defined on $X^{\beth}$ and so when $X$ is not proper we must restrict $\bfp_G$ to $X^{\beth}\subseteq X^{\an}$ for the theorem to hold. See \cite{UlirschLoTrop} or \cite{ThuillierToroidal} for more details on the tropicalization of toroidal embeddings.
\end{remark} 
\noindent
Recall from Section \ref{Section_TheLunaVustTheory} that toroidal $G/H$-embeddings are classified by fans $\fF(X)$ contained in the valuation cone $\calV$. In particular, the $G$-invariant boundary strata are in order-reversing bijection with cones in the fan. We have that $\Sigmabar_X$ is the canonical compactification of $\Sigma_X$, the cone complex of the toroidal embedding $G/H\hookrightarrow X$. It follows from the above Theorem that the cone complex $\Sigma_X$ is equal to $\fF(X)$. In fact, the toroidal strata are equal to the $G$-invariant boundary strata. This all follows from structure theorem for toroidal spherical varieties \cite[Proposition 3.3.2]{Perrin}.

We also have that tropical compactifications of $G/H$ are tropical compactifications in the sense of Ulirsch's tropical compactifications in log-regular varieties \cite{LogTropComp}. 
\begin{proposition}\label{Proposition_PropernessOfIntersectionsInSpherTrop}
    Let $Z$ be a subvariety of $G/H$ and let $X$ be a toroidal $G/H$-embedding such that the closure of $Z$ in $X$, denoted $\overline{Z}$ is a tropical compactification. Then $\overline{Z}$ intersects the boundary strata of $X$ in the expected dimension. To be precise, $\overline{Z}$ intersects the stratum $V$ if and only if the cone associated to that stratum is contained in $\Trop_G(Z)$ and the dimension of the intersection is given by:
    \[
    \dim(\overline{Z}\cap V)=\dim (Z)-\codim(V).
    \]
\end{proposition}
\begin{proof}
    This follows from the fact that for a toroidal $G/H$-embedding, the $G$-invariant boundary strata correspond to the toroidal boundary strata, and the spherical tropicalization map and torpicalization map for toroidal embeddings agree \cite[Theorem C]{ColesUlirsch}, which then means that the results of \cite[Theorem 1.2]{LogTropComp} apply.
\end{proof}

\section{Balancing for spherical tropical curves}\label{Section_balanacingTWO}

Let $C\subseteq G/H$ be an algebraic curve over $k$. As a result of \cite[Theorem 1]{TVSpherTrop} we have that $\Trop_G(C)$ is the support of a fan in $\cV$, and because the spherical tropicalization map agrees with the tropicalization map for toroidal embeddings we have that this fan is dimension at most 1 \cite[Theorem C]{ColesUlirsch}. Because the tropicalization has dimension at most 1 we have that $\Trop_G(C)$ carries a unique fan structure, and this fan defines a tropical compactification of $C$ \cite[Theorem 1.2]{TVSpherTrop}. We will now assign weights to this fan and give the balancing condition. We proceed in an analogous fashion to the definition of weights associated to tropicalization of subvarieties of toroidal embeddings in \cite[Subsection 4.2]{GrossIntersection}. For background on intersection theory we refer the reader to \cite{FultonIntersectionBook} (chapters 6, 7, and 8 in particular), and to \cite{FMSS} for the intersection theory of spherical varieties in particular.

Let $\Sigma$ be a simplicial fan in $\cV$ that covers $\cV$ and contains $\Trop_G(C)$ as a subfan. Let $X$ be the associated toroidal $G/H$-embedding. We have that $X$ is smooth \cite[Corollary 3.3.4]{Perrin} and proper \cite[Theorem 4.2]{Knop}. Such a fan exists because \cite[Theorem C]{Sumihiro} guarantees there is a $G$-equivariant compactification of the toroidal embedding whose associated fan is $\Trop_G(C)$. The associated colored fan of this compactification will contain $\cV$ and contain $\Trop_G(C)$ as a subfan. We can then take the cones in the fan, remove the colors, and intersect the cones with $\calV$ to obtain a fan that covers $\calV$ and contains $\Sigma$, and this fan admits a simplicial refinement if necessary.

\begin{definition}
    Let $\sigma$ be a 1-dimensional cone in $\Trop_G(C)$. We define the \textbf{weight} of $\sigma$, which we denote $m_{\sigma}$, to be:
    \[
    \deg\left([\overline{C}]\cdot [D_{\sigma}]\right).
    \]
    Here $\overline{C}$ is the closure of $C$ in $X$ and $D_{\sigma}$ is the codimension 1 strata associated to $\sigma$. When we write ``$\deg$'' we mean the degree map from the $0$th Chow homology group $A_0(X)$ and the product $[\overline{C}]\cdot [D_{\sigma}]$ is the intersection product.
\end{definition}

\begin{lemma}\label{Lemma_WellDefinednessOfWeight}
    The weights $m_{\sigma}$ are independent of the choice of $\Sigma$.
\end{lemma}

\begin{proof}
    Let $\Sigma'$ be a simplicial refinement of $\Sigma$, and let $X'$ be the respective toroidal $G/H$-embedding.  Then it suffices to show that for each 1-dimensional cone $\sigma$ in $\Trop_G(C)$ that we have:
    \[
     \deg\left([\overline{C}']\cdot [D'_{\sigma}]\right)= \deg\left([\overline{C}]\cdot [D_{\sigma}]\right).
    \]
    Here, $D_{\sigma}'$ is the codimension 1 strata associated to $\sigma$ in $X'$ and $\overline{C}'$ is the closure of $C$ in $X'$. This is sufficient to show because any two fans admit a common simplicial refinement.

    The map of fans $\Sigma'\rightarrow \Sigma$ induced by the identity on $\cV$ induces a proper surjective morphism $f\colon X'\rightarrow X$. In fact this is a morphism of toroidal embeddings. Let $f^*$ and $f_*$ be the pullback and pushforward morphisms on Chow homology, respectively. Then we have the following:
    \[
\deg\left([\overline{C}']\cdot [D'_{\sigma}]\right)=\deg\left(f_*\left([\overline{C}']\cdot [D'_{\sigma}]\right)\right)=\deg\left(f_*\left(f^*[\overline{C}]\cdot [D'_{\sigma}]\right)\right)=\deg\left([\overline{C}]\cdot f_*[D'_{\sigma}]\right)=\deg\left([\overline{C}]\cdot [D_{\sigma}]\right).
    \]
    The first equality follows from the definition of the pushforward, the third equality is the projection formula, and the fourth equality holds because $f$ maps the boundary divisor $D_{\sigma}'$ to $D_{\sigma}$ and is a birational map between the two. The second equality holds because by the definition of $f^*$ we have:
    \[
f^*[\overline{C}']=\gamma_{f}^*(X'\times \overline{C}).
    \]
    Here $\gamma_f$ is the graph morphism from $X'\rightarrow X\times X'$. We have that $\overline{C}$ intersects the strata of $X$ properly by Proposition \ref{Proposition_PropernessOfIntersectionsInSpherTrop} and the morphism $f$ is $G$-equivariant, in particular it maps $C$ isomorphically onto itself because $C$ is contained in $G/H$, so $f^*[\overline{C}]=n[\overline{C}']$ for some $n$, applying the projection forumla to $f_*f^*[\overline{C}]=f_*(f^*[\overline{C}]\cdot[X'])$ shows that $n=1$.
\end{proof}

Let $E_1,\ldots E_r$ be the palette of $G/H$, i.e. the $B$-stable codimension 1 subvarieties.

\begin{definition}
   We define the \textbf{colored weight} associated to $E_j$, which we denote $m_{j}$, to be 
    \[
    \deg\left([\overline{C}]\cdot [\overline{E_j}]\right)
    \]
    where $\overline{C}$ is the closure of $C$ in $X$ and $\overline{E_j}$ is the closure of $E_j$ in $X$.
\end{definition}

\begin{lemma} \label{Lemma_WellDefinednessOfColoredWeights}
    The colored weights are independent of the choice of $X$.
\end{lemma}

\begin{proof}
    Let $X'$ be another $G/H$-embedding given by a simplicial refinement of the fan associated to $X$. Let $\overline{C}'$ and $\overline{E_j}'$ be the closures of $C$ and $E_j$ in $X'$, respectively. Let $f\colon X'\rightarrow X$ be the natural proper $G$-equivariant morphism, then we have the following (as in the proof of Lemma \ref{Lemma_WellDefinednessOfWeight}):
  \[
\deg\left([\overline{C}']\cdot [\overline{E}'_{j}]\right)=\deg\left(f_*\left([\overline{C}']\cdot [\overline{E}'_{j}]\right)\right)=\deg\left(f_*\left(f^*[\overline{C}]\cdot [\overline{E}'_{j}]\right)\right)=\deg\left([\overline{C}]\cdot f_*[\overline{E}'_{j}]\right)=\deg\left([\overline{C}]\cdot [\overline{E}_j]\right).
    \]
with the last inequality holding because $f$ maps the $G$-orbits isomorphically onto itself and thus $E_j$ isomorphically onto itself. Because any two fans admit a common simplicial refinement we have that the numbers $m_j$ are independent of the choice of $X$.
\end{proof}

\begin{remark}
    In \cite[Lemma 3.12]{TropElimTheory} the analogue of lemmas \ref{Lemma_WellDefinednessOfWeight} and \ref{Lemma_WellDefinednessOfColoredWeights} are shown for a toric variety by first showing that given a smooth and proper $X$ such that the closure of $C$ is a tropical compactification, we have that $\overline{C}$ is Cohen-Macaulay at the intersection of $\overline{C}$ with the boundary \cite[Theorem 2.12]{TropElimTheory}. This allows one to show the length of the scheme $\overline{C}\cap D_{\sigma}$ equals the intersection number. In the spherical case it is true that $\overline{C}$ is Cohen-Macaulay at an intersection with the toroidal boundary (even if $C$ is a subvariety of arbitrary dimension) because Cohen-Macaulayness is preserved under etale maps. But $\overline{C}$ could have arbitrarily bad singularities at an intersection point with a color; take, for example, any singular curve in $\bA^{n^2}$. The affine space $\bA^{n^2}$ is a $\GL_n\times \GL_n$ spherical variety and has a color given by the vanishing of $x_{nn}$. Translate the singularity of the curve onto the hyperplane given by the aforementioned color. This gives an arbitrarily bad singularity of a curve intersecting with a color.
\end{remark}

\begin{lemma}\label{lemma_BinvDivs}
    Let $X$ be a toroidal embedding of $G/H$. Let $D$ be a $B$-invariant codimension 1 subvariety of $X$. Then $D$ is either a $G$-invariant boundary divisor, i.e. some divisor $D_{\sigma}$ for $\sigma$ a 1-dimensional cone in the fan of $X$, or $D=\overline{E_j}$ for some $j$.
\end{lemma}

\begin{proof}
    If $D$ intersects $G/H$ then $D\cap G/H=E_j$ for some $j$ and then $D=\overline{E_j}$. If $D$ does not intersect $G/H$ then $GD$ does not intersect $G/H$. So $GD$ has codimension 1 and because $G$ and $D$ are irreducible, then so is $GD$, and so $GD$ is a codimension 1 subvariety of $X$ containing $D$ and thus $GD=D$ so $D$ is $G$-stable.
\end{proof}

Recall that the divisors $D_{\sigma}$ correspond to cones $\sigma$ with primitive integer generator $v_{\sigma}$, which is given by $\varrho(\val_{D_{\sigma}})$. Similarly the elements of the palette $E_j$ define vectors $v_j$ in $N(G/H)$ given by $v_j=\varrho(\val_{E_J})$.

\begin{theorem}\label{theorem_SphericalBalancing}
Let $G/H$ be a spherical homogeneous space and let $C$ be a curve in $G/H$ (defined over $k$).
Then we have that $\Trop_G(C)$ satisfies the following balancing condition:
    \[
    \sum_{\sigma}m_{\sigma}v_{\sigma}=-\sum_{j} m_j v_j
    \]
\end{theorem}

\begin{proof}
Let $\Sigma$ be simplicial fan in $\cV$ which covers $\cV$ and contains $\Trop_G(C)$ as a subfan, and let $X$ be the associated toroidal $G/H$-embedding. Let $\overline{C}$ be the closure of $C$ in $X$. 
Let $d=\dim(X)$.
Let $A_n(X)$ be the Chow homology group of $n$-dimensional cycles. Let $A^n(X)$ be the $n$th Chow cohomology group so that we have a cap product $A^n(X)\otimes A_m(X)\rightarrow A_{m-n}(X)$. Here $X$ is smooth so the Poincaré duality map $A^{n}(X)\rightarrow A_{d-n}(X)$, given by $Z\mapsto Z\cap[X]$, is an isomorphism.
By \cite[Theorem 3]{FMSS} the Kroenecker duality map $A^n(X)\rightarrow \mathrm{Hom}(A_n(X),\bZ)$, given by $a\mapsto (Z \mapsto \deg(a\cap Z))$, is an isomoprhism. Thus each element $Z\in A_{n}(X)$ determines an element of $A^{d-n}(X)$ by acting on an element of $W\in A_{d-n}(X)$ via $\deg([Z]\cdot[W])$. In particular $\overline{C}$ determines an element of $A^{d-1}(X)$, the dual group of $A_{d-1}(X)$ which is the group of divisors. By \cite[Theorem 1]{FMSS} we have that $A_{\bullet}(X)=A^B_{\bullet}(X)$, where $A^B_{\bullet}(X)$ is the Chow homology ring given by $B$-invariant subvarieties whose relations are given by divisors of $B$-semi-invariant rational functions on $B$-stable subvarieties. This means that the class of $\overline{C}$ in $A^{d-1}(X)$ is determined by its values on the $B$-stable codimension 1 subvarieties of $X$, which are exactly the $G$-stable boundary divisors and the divisors $\overline{E_1},\ldots \overline{E_r}$ by Lemma \ref{lemma_BinvDivs}. Because the $B$-invariant rational functions determine the relations on the Chow homology groups we also have that for each function $f\in k(G/H)^{(B)}$ the following must be satisfied:
\[
  \sum_{\sigma\in \Trop_G(C)} m_{\sigma}\textrm{val}_{D_\sigma}(f)+\sum_{j} m_j   \textrm{val}_{E_j}(f)=0.
\]
It then follows that 
   \[
     \sum_{\sigma}m_{\sigma}v_{\sigma} + \sum\limits_j m_j v_j=0.
    \]
\end{proof}

\section{Examples}\label{Section_Examples}

In this section we compute examples of spherical tropicalizations of curves and we check the balancing condition.

\begin{example}\label{example_SL2ActingOnA2}
    Let $G=\SL_2=\spec\left(k[x_{ij}]/(\det(x_{ij})=1)\right)$ and let $H$ be the unipotent subgroup whose diagonal entries are $1$ and lower left entry is $0$. Then cosets in $G/H$ are given by the first column of the matrix and so we identify $G/H$ with $\bA^2\setminus\{(0,0)\}$. Let the coordinate functions on $G/H$ be $x$ and $y$ so that $k(G/H)=k(x,y)$. Let $B$ be the Borel subgroup of upper triangular matrices:
    \[H= \left\{
\begin{bmatrix}
1 & b_{12} \\
0 & 1
\end{bmatrix} \right\},
\quad
B=\left\{
\begin{bmatrix}
b_{11} & b_{12} \\
0 & b_{11}^{-1}
\end{bmatrix} \right\}.
\]
    Then $B$ has an open orbit in $G/H$ given by the complement of the $x$-axis. The group $k(G/H)^{(B)}$ is given by rational functions of the form $\lambda y^n$ where $\lambda\in k^{\times}$ and $n\in \bZ$. If $b\in B(k)$ then $b\cdot \lambda y^n=\chi_1(b) \lambda y^n$ where $\chi_1(b)=b_{11}$. So $M(G/H)\cong \bZ$ and is generated by $\chi_1$. Thus $N_{\R}(G/H)=\R$ and is spanned by $\chi_1^*$, the cocharacter dual to $\chi_1$. There is one divisor, $E_1$, in the palette, it is given by the $x$-axis. We have that $\varrho(\val_{E_1})=\chi_1^*$. The valuation cone consists of all of $N_{\R}(G/H)$, to see this consider the valuation $\val_{\min}$ and $\val_{\max}$ on $k[x,y]$ where $\val_{\min}(f)$ gives the minimum degree of any monomial in $f$ and $\val_{\max}(f)=-\deg(f)$. Both these valuations are $G$-invariant and $\varrho(\val_{\min})=\chi_1^*=-\varrho(\val_{\max})$, thus $\calV=N_{\R}(G/H)$. See Figure \ref{fig:SL2}.
\begin{figure}[ht]
    \centering

\begin{tikzpicture}
  \node at (0,0.3) {};

  \draw[<->] (-3,0) -- (3,0);

\node at (1,0.5) {$\varrho(\val_{E_1})=\chi_1^*$}; 
\node at (0,-0.3) {$0$};

  \foreach \x in {-2,-1,0,1,2}
    \draw (\x,0.1) -- (\x,-0.1);

  \fill[blue] (1,0) circle (2pt);

\end{tikzpicture}
    \caption{Caption: A depiction of $N_{\bR}(G/H)$, $\calV$, and the palette, for $\SL_2/U$.}
    \label{fig:SL2}
\end{figure}
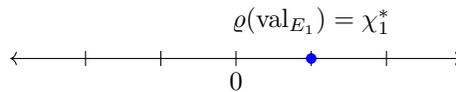

\noindent
    Let $\gamma \in G/H(K)$, then $\trop_G(\gamma)=\min(\val(x(\gamma)),\val(y(\gamma)))\chi_1^*$ (this can be computed directly using coordiantes \cite[Example 2]{TVSpherTrop}). Let $C$ be a degree $d$ curve in $G/H$, then $\Trop_G(C)$ is either the entirety of $N_{\R}(G/H)$ or it is a $1$-dimensional in ray spanned by $\pm\chi_1^*$. The fan in $N_{\R}(G/H)$ given by the cones $\pm\chi_1^*$ (with no colors) corresponds to the blow-up of $\bP^2$ at the origin, which we dnote $X$. Thus for any curve $C$ we have that $X$ is a smooth proper toroidal $G/H$-embedding whose fan contains $\Trop_G(C)$. The cone $\sigma$ generated by $\chi_1^*$ corresponds to the divisors $D_{\sigma}$ in $X$ given by the exceptional divisor, and the divisor $D_{-\sigma}$ corresponds to the hyperplane at infinity. By Bézout's theorem we have that $\deg([\overline{C}]\cdot [D_{-\sigma}])=d$ and if $\deg([\overline{C}]\cdot [\overline{E}_1])=e$ then $\deg([\overline{C}]\cdot [D_{\sigma}])=d-e$. Then the balancing condition for $\Trop_G(C)$ is
    \[
    m_{-\sigma}v_{-\sigma}+ m_{\sigma}v_{\sigma}=d(-\chi_1^*)+(d-e)\chi_1^*=-e(\chi_1^*)=-m_1v_1
    \]
 \end{example}

\begin{example}\label{example_GLNGLNGLN}
Let $\GL_n=\spec\left(k[x_{ij}]_{(\det(x_{ij}))}\right)$
Let $G=\GL_n\times \GL_n$ and let $G$ act on $\GL_n$ by $(g,h)\cdot x =gxh^{-1}$. Then $\GL_2$ is a spherical homogeneous space given by $G/H$ where $H$ is the diagonal subgroup of $G$. Let $B$ be the Borel subgroup of $G$ given by elements $(u,\ell)$ where $u$ is an upper triangular matrix and $\ell$ is a lower triangular matrix. We have that the character lattice of $B$ is given by $\bZ^{2n}$ where $(a,b)=(a_1,\ldots a_n,b_1,\ldots b_n)$ corresponds to the character $\chi_{(a,b)}$ where
\[
\chi_{(a,b)}(u,\ell)=\prod_{i=1}^{n}u_{ii}^{a_i}b_{ii}^{\ell_i}.
\]
The group $k(G/H)^{(B)}$ is generated by the following $B$-semi-invariant functions:
\[
h_i=\textrm{det}\left[ {\begin{array}{cccc}
   x_{ii} & x_{i(i+1)} & \cdots & x_{in}  \\
   x_{(i+1)i} & x_{(i+1)(i+1)} &  & \vdots  \\
   \vdots &  & \ddots & \\
   x_{ni} & \cdots & & x_{nn}\\
  \end{array} } \right].
\] 
We have that $\chi_{h_i}=\chi_{(-a_i,a_i)}$ where $a_i\in \bZ^n$ is equal to $(0,\ldots 1,\ldots 1)$ where the first nonzero entry is the $i$th entry. The set $\cD(G/H)$ contains $n-1$ elements given by the vanishing of the functions $h_2,\ldots h_n$. 

We will now give more convenient coordinates for $M(G/H)$ and $N(G/H)$. Let $f_i=h_i/h_{i+1}$ for $i<n$, for $i=n$ set $f_i=h_n=x_{nn}$. We have that $\chi_{f_i}=\chi_{(-e_i,e_i)}$ where $e_i$ is the element of $\bZ^n$ that has a $1$ in the $i$th position and $0$s elsewhere. Set $\chi_i=\chi_{f_i}$. So $\chi_1,\ldots \chi_n$ forms a basis for $M(G/H)$. Let $\chi_i^*$ be a dual basis for $N(G/H)$. The valuation cone $\calV\subseteq N_{\R}(G/H)$ consists of points $(\mu_1,\ldots \mu_n)$ where $\mu_1\geq \mu_2\cdots \geq \mu_n$. The spherical tropicalization map $\trop_G$ is given by the `Cartan Decomposition' of the matrix, also known as the `invariant factors' \cite[Theorem 2]{TVSpherTrop}. Let $K=k\{\!\{t\}\!\}$ be the field of Puiseux series, equipped with the $t$-adic valuation. Then a matrix in $\gamma\in \GL_n(K)$ can be written as a product $gth$ where $g,h\in \GL_n(\cO_K)$ and $t\in T(K)$ such that:
\[
t=\left[ {\begin{array}{cccc}
   t^{\mu_1} & 0 & \cdots & 0  \\
   0 & t^{\mu_2}&  & \vdots  \\
   \vdots &  & \ddots & \\
   0 & \cdots & & t^{\mu_n}\\
  \end{array} } \right]
\]
and $\mu_1\geq \mu_2 \cdots \geq\mu_n$. Here $\calO_K$ is the valuation ring of $K$. The element $t$ will be unique and $\trop_G(\gamma)=(\mu_1,\ldots \mu_n)\in \calV$. \footnote{For any valued field $K$ we can write an element $g\in \GL_n(K)$ as $gth$ where $g,h\in G(\cO_K)$, $t\in T(K)$, and the valuations of the diagonal elements of $t$ are decreasing. The choice of $t$ may not be entirely unique, but the valuations of the diagonal elements will be uniquely determined.}

\begin{figure}[H]
\centering 

\begin{tikzpicture}[scale=0.7,baseline={(0,0)}]

    \fill[red!60, opacity=0.5] (-3,-3) -- (3,-3) -- (3,3) -- cycle;

\foreach \x in {-3,-2,-1,0,1,2,3}
        \foreach \y in {-3,-2,-1,0,1,2,3}
            \fill[black] (\x,\y) circle (1.5pt);

    \draw[->] (-3,0) -- (3,0) node[right] {$\chi_1^*$};
    \draw[->] (0,-3) -- (0,3) node[above] {$\chi_2^*$};

    \node at (3.3, 2.75) {$\cV$};
    \fill[blue] (-1,1) circle (3pt) ;
    \node at (-1, 1.5) {$\varrho(\val_{E_1})$};

\end{tikzpicture}
\begin{tikzpicture}[scale=0.7,baseline={(0,0)}]

    \fill[red!60, opacity=0.5] (-3,-3) -- (3,-3) -- (3,3) -- cycle;

\foreach \x in {-3,-2,-1,0,1,2,3}
        \foreach \y in {-3,-2,-1,0,1,2,3}
            \fill[black] (\x,\y) circle (1.5pt);

    \draw[->] (-3,0) -- (3,0)  ;
    \draw[->] (0,-3) -- (0,3)  ;

    \fill[blue] (-1,1) circle (3pt) ;

    \draw[->,line width=0.75mm] (0,0) -- (3,0) node[right]{$\sigma_2$} ;
    \draw[->, line width=0.75mm] (0,0) -- (-3,-3) node[below]{$\sigma_1$} ;
    \draw[->, line width=0.75mm] (0,0) -- (3,3)  node[above]{$-\sigma_1$};
    \node at (4, 1.5) {$\fF(X)$};
    
    \foreach \x in {0.5,1,1.5,2,2.5,3} {
        \draw (-\x,-\x) -- (\x,0);   
    }
    \draw (-2.5,-3) -- (3,-0.25);
    \draw (-2,-3) -- (3,-0.5);
    \draw (-1.5,-3) -- (3,-0.75);
    \draw (-1,-3) -- (3,-1);
    \draw (-0.5,-3) -- (3,-1.25);
    \draw (0,-3) -- (3,-1.5);
   \draw (0.5,-3) -- (3,-1.75);
   \draw (1,-3) -- (3,-2);
\draw (1.5,-3) -- (3,-2.25);
\draw (2,-3) -- (3,-2.5);
\draw (2.5,-3) -- (3,-2.75);

     \foreach \x in {0.5,1,1.5,2,2.5,3} {
        \draw (\x,0) -- (\x,\x);    
    }

\end{tikzpicture}

    \caption{On the left The vector space $N_{\R}(G/H)$, valuation cone $\calV$, the color $\varrho(\val_{E_1})$, for $n=2$. On the right, the fan of $X$.
    }
    \label{fig_valcone}
\end{figure}
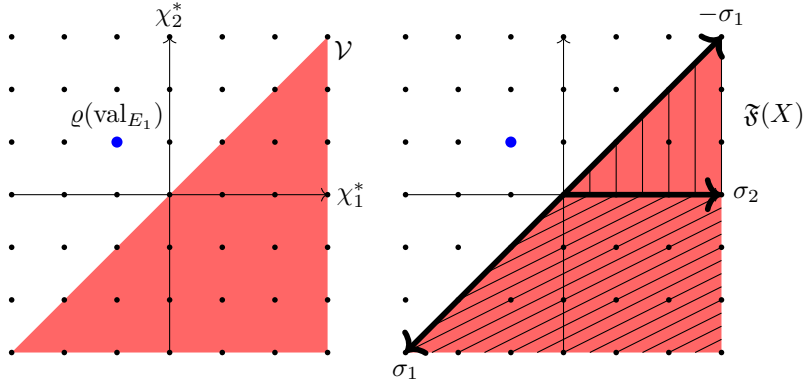

Let $n=2$, so there is one color $E_1$ given by the vanishing $x_{22}$. Consider the line $C=V(x_{11}-x_{12}-1, x_{12}-x_{21},x_{22})$. That is, matrices of the form:
\[
\begin{bmatrix}
g_{12}+1 & g_{12} \\
g_{12} & 0
\end{bmatrix}.
\]
As in \cite[Example 4]{TVSpherTrop} we have that $\Trop_G(C)$ is given by the cones $\sigma_1=\{(\mu_1,\mu_2)\mid \mu_1=\mu_2\leq 0\}$ and $\sigma_2=\{(\mu_1,\mu_2)\mid \mu_1\geq0\textrm{ and } \mu_2=0\}$.
We can embed $\GL_2$ into $\bP^4$ via $g\mapsto [g_{11}:g_{12}:g_{21}:g_{22}:1]$ and if $X$ is the blow-up of $\bP^4$ at the origin then this extends to an embedding of $\GL_2$ into $X$. The action of $G$ extends to $X$ and $X$ is a proper smooth toroidal $G/H$-embedding, and the fan of $X$ contains $\trop_G(C)$. The fan of $X$ consists of the origin, the 1-dimensional cones $\sigma_1$ and $\sigma_2$ as well as the cone given by $-\sigma_1$, and the two 2-dimensional cones between these 1-dimensional cones. See Figure \ref{fig_valcone} for a depiction.

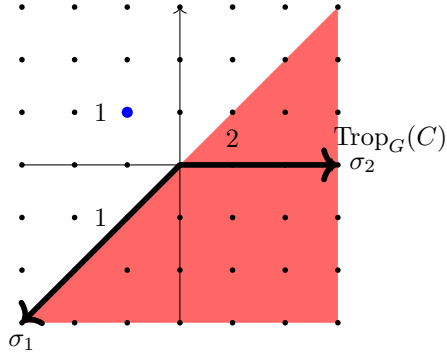
\begin{figure}[H]
\centering

\begin{tikzpicture}[scale=0.7]

    \fill[red!60, opacity=0.5] (-3,-3) -- (3,-3) -- (3,3) -- cycle;

\foreach \x in {-3,-2,-1,0,1,2,3}
        \foreach \y in {-3,-2,-1,0,1,2,3}
            \fill[black] (\x,\y) circle (1.5pt);

    \draw[->] (-3,0) -- (3,0) ;
    \draw[->] (0,-3) -- (0,3) ;

    \fill[blue] (-1,1) circle (3pt) ;

 \draw[->,line width=0.75mm] (0,0) -- (3,0) node[right]{$\sigma_2$} ;
    \draw[->, line width=0.75mm] (0,0) -- (-3,-3) node[below]{$\sigma_1$} ;
    \node at (4, 0.5) {$\Trop_G(C)$};
    \node at (1, 0.5) {2};
    \node at (-1.5,-1) {1};
    \node at (-1.5,1) {1};
    
\end{tikzpicture}

    \caption{The spherical tropicalization of $C$ with its weights.
    }
    \label{fig_tropicalization}
\end{figure}

The 1-dimensional cones in $\fF(X)$ correspond to divisors in $X$. The divisor $D_{\sigma_1}$ is the the hyperplane at $\infty$. The divisor $D_{\sigma_2}$ is the vanishing of the determinant. Then we can compute that $m_{\sigma_1}=1$ as $C$ is a line and $D_{\sigma_1}$ is a hyperplane, and $m_{\sigma_2}=2$ as the determinant of a matrix in $g\in C$ is equal to $g_{12}^2$. There is one color $E_1$ given by the vanishing of $x_{22}$ and so the colored weight $m_1$ is $1$ as $E_1$ is a hyerplane. We than have that:
\[
m_{\sigma_1}v_{\sigma_1}+m_{\sigma_2}v_{\sigma_2}=1(-1,-1)+2(1,0)=-1(-1,1)=-m_{1}v_1
\]
so $\Trop_G(C)$ satisfies the balancing condition. See Figure \ref{fig_tropicalization} for a depiction.
\end{example}

\section{Closing remarks on a general structure theory}
\label{section_closingremarks}

We will now discuss some related work on balancing for tropicalizations of subvarieties of toroidal embeddings, other work on the structure on spherical tropicalizations, and how the techniques in the previous section may generalize to higher dimensions.

We work with toroidal $G/H$-embeddings. In \cite{GrossIntersection} Gross gives a balancing condition for the tropicalization of subvarieties of toroidal embeddings, and we can apply those results to spherical tropicalization because we have that: the $G$-invariant strata are the toroidal strata, the spherical tropicalization map and tropicalization map for toroidal embeddings agree, and spherical tropical compactifications are tropical compactifications in log-regular varieties (see Sections \ref{Section_toroidalembeddings}). We now describe what conditions this imposes on the spherical tropicalization.
Gross introduces intersection theory for tropicalizations of subvarieties of toroidal embeddings by mapping the cone complex of a toroidal embedding into a vector space $N_{\R}^X$ \cite{GrossIntersection}. Set 
\[
M^X\colon= \cO^{\times}_{X}(G/H)/k^{\times}
\]
 and let $N^X$ be the dual lattice. Let $V_i$ be the subgroup of $M(G/H)$ given by functions in $k(G/H)^{(B)}$ that do not vanish (or have a pole) on the color $E_i$.

 \begin{lemma}
     There is an inclusion $M^X\hookrightarrow M(G/H)$ and the image is $\cap_{i=1}^rV_i$.
 \end{lemma}

 \begin{proof}
     Let $f\in \cO^{\times}_{X}(G/H)$, and fix $x$ in the open $B$-orbit of $G/H$. Then the map $b\mapsto b \cdot f(x)$ determines a character of $B$ because $f$ is invertible. Denote this character by $\chi$, by our choice of $x$ we have that $b\cdot f=\chi(b)f$, so $f$ is $B$-semi-invariant.  This defines the inclusion.

     To see that the image of the inclusion is $\cap_{i=1}^rV_i$ it suffices to observe that $f$ does not vanish on the open $B$-orbit, and if $f$ does have not have zero or poles on the divisors $E_1,\ldots E_r$ then $f$ is invertible on $G/H$, and conversely if $f$ is invertible on $G/H$ it does not have any zero or poles on the divisors $E_1\ldots E_r$.
 \end{proof}

\noindent The inclusion $M^X\hookrightarrow M(G/H)$ induces a surjection $N_{\R}(G/H)\rightarrow N_{\R}^X$. This surjection is given by quotienting the vector $N_{\R}(G/H)$ by the vectors $v_1,\ldots v_r$ (recall that $v_j=\varrho(\val_{E_j})$. Recall that the cone complex $\Sigma_X$ is canonically identified with the fan in $\calV\subseteq N_{\R}(G/H)$. There is a `weak embedding' $\varphi_X\colon \Sigma_X\rightarrow N^X_{\R}$ which is continuous and integral linear on the cones of $\Sigma_X$ as described in \cite[Subsection 2.3]{GrossIntersection}.

\begin{proposition}\label{prop_balancingweaklyembedded}
    The surjection $N_{\R}(G/H)\rightarrow N_{\R}^X$ agrees with the continuous map $\varphi_X\colon \Sigmabar_X\rightarrow N^X_{\R}$.
\end{proposition}

\begin{proof}
    By definition the map $\Sigmabar_X\rightarrow N^X_{\R}$ is given by restricting to to elements of $M^X$.
\end{proof}
\noindent 
Let $Z\subseteq G/H$ be an $n$-dimensional variety. We can find a smooth, proper, and toroidal $G/H$-embedding $X$, such that the support of the fan of $X$ contains $\Trop_G(Z)$, we can continue the follow the approach of \cite{GrossIntersection}. 
Let $\Delta_0$ be a fan whose support is $\Trop_G(Z)$ and let $\Delta$ be a fan containing $\Delta_0$ such that the associated toroidal embedding $G/H\subseteq X$ is smooth and proper. Then for each $n$-dimensional cone $\sigma \in\Delta_0$ we can assign a weight to $\sigma$ by defining:
\[
m_{\sigma}\colon = \deg\left([\overline{Z}]\cdot[V_{\sigma}]\right)
\]
where $V_{\sigma}$ denotes the closure of the codimension $d$ boundary stratum associated to $\sigma$ and $\overline{Z}$ is the closure of $Z$ in $X$. The toroidal boundary strata are the $G$-invariant boundary strata so by \cite[Proposition 4.7]{GrossIntersection} the weights are well-defined, independent of the choice of $\Delta$ and $\Trop_G(Z)$ is balanced in the sense that the image of the weighted fan $\Delta_0$ under the map $\phi_X$ is a balanced fan (i.e. the tropicalization is balanced in the sense of balancing for weakly embedded complexes \cite[Definition 3.1]{GrossIntersection}). In particular Proposition \ref{prop_balancingweaklyembedded} gives that $\Trop_G(Z)$ is balanced after taking a quotient of $N_{\R}(G/H)$ be the image of the palette.

As in the curve case one would hope that we can refine the above condition to include all the information coming from the Chow cohmology, as the above construction only records intersection with with the toroidal strata which will in general be different than the $B$-orbit closures (take any variety with non-zero palette, for example). To mimic the toric case for a higher dimensional generalization one would want to assign weights, as above, to cones of dimension $n$ if $Z$ is dimension $n$, and balancing would come from analyzing the behavior of $\Trop_G(Z)$ around $n-1$ dimensional cones. Fix an $(n-1)$-dimensional cone $\tau$ in $\Trop_G(Z)$,
and let $\sigma_1,\ldots \sigma_p$ be the $n$-dimensional cones containing $\tau$. Then the image of $\sigma_1,\ldots \sigma_p$ in $N_{\R}(\tau)$ are $1$-dimensional cones and one could derive balancing conditon on these 1-dimensional cones similar to Theorem \ref{mainthm_balancingforcurves}. There are however several points of departure from the toric case. For one, this still does not capture all the information in the Chow chomology class of $Z$ as it only compares intersection with $B$-invariant subvarieties that are contained in closures of boundary strata corresponding to the maximal dimensional cones of $\Trop_G(Z)$, this is not all of the $B$-invariant subvarieties; for example this does not record intersection with $B$-invariant subvarieties in the interior of $G/H$. There is also the issue that in general $\Trop_G(Z)$ will not be pure dimensional or locally connected in codimension 1 \cite[Example 5,  Figure 4(b)]{TVSpherTrop}, and in general an $n$-dimensional variety does not necessarily tropicalize to something with $n$-dimensional cones \cite[Example 5,  Figure 4(a)]{TVSpherTrop}. So it is less clear that only the $n$-dimensional cones should be assigned weights

In Nash's thesis here are some other results on the structure of $\Trop_G(Z)$ \cite{NashThesis}. In \cite{NashGlobal} it is shown that spherical tropicalization can be computed by first embedding $X$ into a toric variety and applying the tropicalization map for tori. These same ideas are used to show that $\Trop_G(Z)$ is balanced after projection to a linear subspace and removal of maximal cones which are not of maximum dimension \cite[Theorem 5.2.4]{NashThesis}.

We would like to end by discussing the case of non-trivial valuation. For a subvariety $Z$ defined over $K$ a non-trivially valued field extension of $k$ i.e. the case of ``non-constant coefficients" we still can define a tropicalization as the image of the composition $Z^{\an}\hookrightarrow (G/H)_K^{\an}\rightarrow G/H^{\an}\rightarrow \calV$. Here the second map is the base change map for Berkovich spaces and the last map is $\trop_G$.  We have that the tropicalization of $Z$ in a toric variety locally looks like the support of a balanced fan, that is, $\Trop_T(Z)$ locally looks like the tropicalization of a subvariety defined with trivially valued coefficients. Howver, one can see that if $\calV$ is strictly convex then it is impossible that every point of $\trop_G(Z)$ the tropicalization locally looks like $\trop_G(W)$ where $W$ is a subvariety of $G/H$ defined over $k$.  We suspect that a precise statement of the local structure of $\Trop_G(Z)$ requires understanding of degenerations of $G/H$, possibly drawing on the Gröbner theory of spherical varieties \cite{KMGrob}.
 In \cite[Theorem 5.4.1]{NashGlobal} Nash shows that $\trop_G(Z)$ is the support of a polyhedral complex and that the recession fan of $\Trop_G(Z)$ is determined by tropicalizing $Z$ after changing the valuation on $K$ to the trivial valuation \cite[Theorem 5.4.5]{NashThesis}.

\nocite{*}
\bibliographystyle{amsalpha}
\bibliography{BalancedSpherTrop}{}

\end{document}